\documentclass[12pt]{amsart}

\usepackage{amssymb,amsmath,amsthm, color}
\usepackage{xcolor}
\usepackage{geometry}
\usepackage{verbatim}
\usepackage[normalem]{ulem}

\def\Ex{{\mathbb E}}
\def\E{{\mathbb E}}
\def\Pr{{\mathbb P}}
\def\er{{\mathbb R}}
\def\R{{\mathbb R}}
\def\calz{\mathcal{Z}}

\newcommand{\scal}[2]{\left\langle #1, #2\right\rangle}

\newtheorem{thm}{Theorem}
\newtheorem{prop}[thm]{Proposition}
\newtheorem{lem}[thm]{Lemma}
\newtheorem{cor}[thm]{Corollary}

\theoremstyle{remark}
\newtheorem{rem}[thm]{Remark}

\theoremstyle{definition}
\newtheorem{dfn}[thm]{Definition}

\title{Hadamard products and moments of random vectors}

\author{Rafa{\l} Lata{\l}a }
\author{Piotr Nayar}
\address{Institute of Mathematics \\ University of Warsaw\\ Banacha 2, 02-097, Warsaw, Poland}
\email{rlatala@mimuw.edu.pl, nayar@mimuw.edu.pl}

\thanks{The authors were supported by the National Science Centre, Poland, grant 2015/18/A/ST1/00553. P.N. was supported in part by the National Science Centre, Poland, grant 2018/31/D/ST1/01355.}


\begin{document}

\newgeometry{tmargin=3cm, bmargin=3cm, lmargin=3cm, rmargin=3cm}

\maketitle

\begin{abstract}
We derive new comparison inequalities between weak and strong moments of norms of random vectors with optimal
(up to an universal factor) constants. We discuss applications to the concentration of log-concave random vectors and bounds on $p$-summing norms of finite rank operators.
\end{abstract}

\section{Introduction}

The study of moments of random variables is an essential issue of probability theory, one of the reasons being the fact that tail estimates for random variables are related
to bounds for their moments via the Markov inequality. In probabilistic convex geometry and concentration of measure theory one is  often interested in bounding $p$th moments of random vectors. 

To be more precise, the $p$th \emph{strong moment} of a random vector $X$ in $\R^n$ with respect to a given norm structure $(\R^n,\|\cdot\|)$ is defined as $M_{p}(X) = (\E\|X\|^p)^{1/p}$. Another related quantity is the so-called \emph{weak} $p$th \emph{moment} defined as $\sigma_{p}(X) = \sup_{\|t\|_\ast \leq 1}(\E |\scal{t}{X}|^p )^{1/p}$, where $\|\cdot\|_\ast$ denotes the dual norm.   
\noindent Weak moments are usually much easier to compute or estimate, and so comparison inequalities between weak and strong moments are of interest in convex geometry, see e.g. \cite{BGVV}. While the one-sided estimate $\sigma_{p}(X) \leq M_{p}(X)$ follows trivially from the fact that $\|x\|=\sup_{\|t\|_\ast \leq 1} |\scal{t}{x}|$, obtaining the reverse bounds turns out to be much more challenging. As an example let us mention the Paouris inequality $M_{p}(X) \leq C (M_1(X)+ \sigma_p(X))$ valid for the standard Euclidean norm
and arbitrary log-concave random vector $X$ in $\R^n$, see \cite{Pa} and \cite{ALOPT} (see also \cite{LS2} for an extension of this result to a larger class of norms).  Here and in the sequel $C$ denotes an 
absolute constant, whose value may change at each occurrence.

Usually, to compare weak and strong moments one applies the concentration of measure theory \cite{Le} or the chaining method \cite{TaMon}. What is crucial for such proofs is the regularity of the random vector $X$  and/or the special form of the norm.    

One may however wonder what is the minimal constant $C_{n,p}$ such
that the inequality $M_p(X) \leq C_{n,p} \sigma_p(X)$ holds for any $n$-dimensional random vector $X$ and any 
 norm on $\er^n$.
First attempt to this problem was discussed in \cite{La}, where it was proved that for
unconditional random vectors such comparison holds with $C_{n,p}=C\sqrt{(n+p)/p}$. Our main result shows
that this inequality is valid without any symmetry assumptions.      

\begin{thm}
\label{thm_main}
For any $n$-dimensional random vector $X$ and any nonempty set $T$ in $\er^n$ we have
\begin{equation}\label{eq:main}
\left(\Ex\sup_{t\in T}|\langle t,X\rangle|^p\right)^{1/p}
\leq \ 2\sqrt{e} \sqrt\frac{n+p}{p}  \  \sup_{t\in T}\left(\Ex|\langle t,X\rangle|^p\right)^{1/p}\quad \mbox{ for }p\geq 2.
\end{equation}
In particular, for any normed space $(\R^n,\|\cdot\|)$ we have 
\[
(\E\|X\|^p)^{1/p}
\leq 2\sqrt{e} \sqrt\frac{n+p}{p} \sup_{\|t\|_*\leq 1}\left(\E|\langle t,X\rangle|^p\right)^{1/p}\quad \mbox{ for }p\geq 2.
\]
\end{thm}

\noindent The proof of the main result uses elementary linear algebra, namely a Hadamard power trick inspired by the proof of the so-called Welch bound (see \cite{Welch}) given in  \cite{DSD}. 

\begin{rem} 
\label{rem_zp}
Let us notice that by homogeneity one can always assume that the supremum on the right hand side of \eqref{eq:main} is one. Then by enlarging the set $T$ we may assume that $T$ is the set of all vectors $t$ satisfying $\E|\scal{t}{X}|^p \leq 1$. Thus, inequality  \eqref{eq:main} may be equivalently stated as
\begin{equation}
\label{eq:momZp}
\left(\Ex\|X\|_{\calz_p(X)}^p\right)^{1/p}\leq 2\sqrt{e} \sqrt\frac{n+p}{p}, 
\end{equation}
where
\[
\|s\|_{\calz_p(X)}:=\sup\{|\langle t,s\rangle|\colon\ \Ex|\langle t,X\rangle|^p\leq 1\}.
\]
This has been conjectured (with a universal constant in place of $2\sqrt{e}$) by the second named author in  \cite{La} (see Problem 1 therein). Thus Theorem \ref{thm_main} positively resolves this conjecture.
\end{rem}

\begin{rem}
It is not hard to check (see \cite{La} for details) that for rotationally invariant $n$-dimensional random vector 
$X$ with finite $p$th moment one has 
\[
\left(\Ex\|X\|_{\calz_p(X)}^p\right)^{1/p}=\left(\Ex|U_1|^p\right)^{-1/p}\sim \sqrt\frac{n+p}{p},
\]
where $U_1$ is the first coordinate of the random vector $U$ uniformly distributed on the unit sphere. Therefore the bound in Theorem \ref{thm_main} is sharp up to a universal multiplicative factor.
\end{rem}

\begin{rem}
In general one cannot reverse bound \eqref{eq:momZp} for $2\ll p\ll n$. Indeed, let $e_1,\ldots,e_n$ be the canonical basis of $\er^n$ and $\Pr(X=\pm e_i)=1/(2n)$ for $1\leq i\leq n$. Then for $s,t\in \er^n$,
\[
\Ex|\langle t,X\rangle|^p=\frac{1}{n}\sum_{i=1}^n|t_i|^p,\quad
\|s\|_{\calz_p(X)}=n^{1/p}\left(\sum_{i=1}^n|s_i|^q\right)^{1/q},
\]
where $q$ denotes the H\"older dual to $p$. Thus for $2\ll p\ll n$,
\[
\left(\Ex\|X\|_{\calz_p(X)}^p\right)^{1/p}=n^{1/p}\ll \sqrt\frac{n+p}{p}.
\]
We think however that \eqref{eq:momZp} may be reversed in the case of symmetric log-concave vectors (see the discussion in Section \ref{sec:openq}).
\end{rem}

Inequality \eqref{eq:main} arose as a result of investigating optimal concentration of measure inequalities. The discussion of this and other applications will be given is Section \ref{sec:apl}. Section \ref{sec:proof} is devoted to the proof of Theorem \ref{thm_main}. We conclude with reviewing some related open problems in Section 
\ref{sec:openq}.

\section{Proof of the main result}\label{sec:proof}

The following Lemma gives three equivalent formulations of our main inequality. 

\begin{lem}
\label{lem:main}
Let $1\leq p<\infty$ and $n$ be a positive integer. The following conditions are equivalent:\\
i) For any $n$-dimensional random vector $X$ and any nonempty set $T$ in $\er^n$ we have
\[
\left(\Ex\sup_{t\in T}|\langle t,X\rangle|^p\right)^{1/p}
\leq C_{n,p}\sup_{t\in T}\left(\Ex|\langle t,X\rangle|^p\right)^{1/p}.
\]
ii) For any $k,l\geq 1$ and any vectors $t_1,\ldots,t_k$ and $x_1,\ldots,x_l$ in $\er^n$ we have
\[
\left(\sum_{j=1}^l\sup_{1\leq i\leq k}|\langle t_i,x_j\rangle|^p\right)^{1/p}
\leq C_{n,p}\sup_{1\leq i\leq k}\left(\sum_{j=1}^l|\langle t_i,x_j\rangle|^p\right)^{1/p}.
\]
iii) For any $k,l\geq 1$ and any matrix $(a_{ij})_{i\leq k,j\leq l}$ of rank at most $n$ we have
\[
\left(\sum_{j=1}^l\sup_{1\leq i\leq k}|a_{ij}|^p\right)^{1/p}
\leq C_{n,p}\sup_{1\leq i\leq k}\left(\sum_{j=1}^l|a_{ij}|^p\right)^{1/p}.
\]
\end{lem}

\begin{proof}
i)$\iff$ ii) Simple approximation argument shows that it is enough to show i) for a finite set $T$ and a random vector $X$ equidistributed over a finite subset of $\er^n$.

ii)$\iff$iii)  Matrix $A:=(\langle t_i,x_j\rangle)_{i\leq k,j\leq l}$ is a product $TX$, where $T$ is a $k \times n$ matrix whose rows are the vectors $t_i$ and $X$ is an $n \times l$ matrix whose columns are the vectors $x_j$. Thus it is a matrix of a linear map $\R^l \to \R^n \to \R^k$ and so it has rank at most $n$.   Moreover any $k \times l$ matrix of rank at most $n$ may be represented as such a product $TX$ due to the rank factorization theorem.  
\end{proof}


The proof of Theorem \ref{thm_main} is divided into three steps.

\begin{prop} Conditions i)-iii) hold
\label{lem:p=2k}
\begin{itemize}
	\item[(a)]  for $p=2$ with $C_{n,2}=\sqrt{n}$,  
	\item[(b)]  for $p=2m$, $m=1,2,\ldots$ with $C_{n,2m}=\binom{n+m-1}{m}^{1/(2m)}$,
	\item[(c)]  for $p \in (2m,2m+2]$, $m=1,2,\ldots$ with 
                    $C_{n,p}=C_{n,2m}=\binom{n+m-1}{m}^{1/(2m)}\leq 2\sqrt{e}\sqrt\frac{n+p}{p}$.
\end{itemize}

\end{prop}

\begin{proof}
(a) We will show condition i).
Without loss of generality we may assume that a vector $X$ is symmetric, bounded and has a nondegenerate covariance
matrix $C$. Let $\alpha:=\sup_{t\in T}\Ex|\langle t,X\rangle|^2=\sup_{t\in T}\langle Ct,t\rangle$. Then
\begin{align*}
\Ex\sup_{t\in T}|\langle t,X\rangle|^2& \leq \Ex\sup\{|\langle s,X\rangle|^2\colon \langle Cs,s\rangle\leq \alpha\} \\
& = \Ex\sup\{|\langle C^{1/2}s,C^{-1/2} X\rangle|^2\colon  |C^{1/2}s|^2 \leq \alpha\} 
 = \alpha\Ex |C^{-1/2}X|^2  =\alpha n.
\end{align*}

(b) We will establish condition iii).
Let $A$ be an $k \times l$ matrix of rank at most $n$. Then by  the rank factorization theorem $A$ can be written in the form $A=TX$, where $T$ is $k \times n$ and $X$ is $n \times l$. Then  $A=\sum_{r=1}^n u_r\otimes v_r$, where $u_r$ in $\R^k$ are column vectors of $T$ and $v_r$ in $\R^l$ are row vectors of $X$, $r=1,\ldots,n$. Then the $m$th Hadamard power of $A$ equals
\begin{align*}
A^{\circ m}  =\left(\sum_{r=1}^n u_r\otimes v_r\right)^{\circ m}
&  =\sum_{r_1,\ldots,r_m=1}^n \left(u_{r_1}\otimes v_{r_1}\right)\circ\ldots\circ \left(u_{r_m}\otimes v_{r_m}\right)
\\ 
& =\sum_{r_1,\ldots,r_m=1}^n  \left(u_{r_1}\circ\ldots\circ u_{r_m}\right)\otimes  
\left(v_{r_1}\circ\ldots\circ v_{r_m}\right).
\end{align*}
Observe that Hadamard product is commutative so we may restrict the sum to nondecreasing $m$-tuples 
$(r_1,\ldots,r_m)$ with values in $\{1,\ldots,n\}$. Number of such tuples is $\binom{n+m-1}{m}$. Thus
matrix $A^{\circ m}$ has rank at most $\binom{n+m-1}{m}$ (see \cite{peng} for a similar arguments for Hadamard powers of Gram matrices). We conclude by using part (a) for matrix $A^{\circ m}$.

(c) We will verify condition ii). Let $p_j=\max_{i\leq k}|\langle t_i,x_j\rangle|^{p-2m}$, by the homogenity we may
assume that $\sum_{j=1}^l p_j=1$. Let $X$ be a random variable such that $\Pr(X=x_j)=p_j$, Lemma \ref{lem:p=2k}
yields
\begin{align*}
\sum_{j=1}^l\sup_{1\leq i\leq k}&|\langle t_i,x_j\rangle|^p
=\Ex\sup_{1\leq i\leq k}|\langle t_i,X\rangle|^{2m}
\leq C_{n,2m}^{2m}\sup_{1\leq i\leq k}\Ex|\langle t_i,X\rangle|^{2m}
\\
&=C_{n,2m}^{2m}\sup_{1\leq i\leq k}\sum_{j=1}^{l}|\langle t_i,x_j\rangle|^{2m}p_j
\\
&\leq C_{n,2m}^{2m}\sup_{1\leq i\leq k}\left(\sum_{j=1}^{l}|\langle t_i,x_j\rangle|^{p}\right)^{2m/p}
\left(\sum_{j=1}^lp_j^{p/(p-2m)}\right)^{(p-2m)/p}
\\
&=C_{n,2m}^{2m}\sup_{1\leq i\leq k}\left(\sum_{j=1}^{l}|\langle t_i,x_j\rangle|^{p}\right)^{2m/p}
\left(\sum_{j=1}^l\sup_{1\leq i\leq k}|\langle t_i,x_j\rangle|^p\right)^{(p-2m)/p},
\end{align*}
where the second inequality follows by the H\"older inequality. After rearranging we get the desired bound. 

To conclude observe that
\[
C_{n,2m}^2\leq \frac{e(n+m-1)}{m}\leq e\frac{n+p/2}{p/4}\leq 4e\frac{n+p}{p}.
\]
\end{proof}


\section{Applications}\label{sec:apl}

\subsection*{Concentration inequalities} 

Let $\nu$ be a symmetric exponential measure with parameter $1$, i.e.\ the measure on the real line 
with the density $\frac{1}{2}e^{-|x|}$. Talagrand \cite{Ta} showed that the product measure $\nu^n$
satisfies the following two-sided concentration inequality
\[
\forall_{A\in \mathcal{B}(\er^n)}\ \forall_{p>0}\
\nu^n(A)\geq \frac{1}{2}\ \Rightarrow\  
1-\nu^n(A+C\sqrt{p}B_2^n+CpB_1^n)\leq e^{-p}(1-\nu^n(A)).
\]
where by a letter $C$ here and in the sequel we denote universal constants.

This is a remarkably strong concentration result implying, for example, the celebrated concentration of measure phenomenon for the canonical Gaussian measure $\gamma_n$ on $\er^n$:
\[
\forall_{A\in \mathcal{B}(\er^n)}\ \forall_{p>0}\
\gamma_n(A)\geq \frac{1}{2}\ \Rightarrow\  
1-\gamma_n(A+C\sqrt{p}B_2^n)\leq e^{-p}(1-\gamma_n(A)),
\]
discovered (in the sharp isoperimetric form) by Sudakov and Tsirelson in \cite{SC74}, and independently by Borell in \cite{B75}.

It is not hard to check that $\calz_p(\nu^n)\sim \sqrt{p}B_2^n+pB_1^n$ and $\calz_p(\gamma_n)\sim \sqrt{p}B_2^n$
for $p\geq 2$, where
for a probability measure $\mu$ on $\er^n$ and a random vector $X$ distributed according to $\mu$ we set 
\[
\calz_p(\mu)=\calz_p(X)=\{t\in \er^n\colon\ \|t\|_{\calz_p(X)}\leq 1\}.
\]

In the context of convex geometry it is natural to ask if similar inequalities hold for other log-concave measures, namely measures with densities of the form $e^{-V}$, where $V:\R^n \to R$ is convex.  
An easy observation from \cite{LW} shows that if $\mu$ is a symmetric log-concave probability measure
and $K$ is a convex set such that for any halfspace $A$ satisfying $\mu(A)\geq \frac{1}{2}$ we have $\mu(A+K)\geq 1- \frac{1}{2}e^{-p}$, then necesarily $K\supset c\calz_{p}$. This motivates the following definiton proposed in \cite{LW} by Wojtaszczyk and the second named author. 

\begin{dfn}
We say that a measure $\mu$ satisfies the {\em optimal concentration inequality with
constant $\beta$} (\! $\mathrm{CI}(\beta)$ in short) if for any Borel set $A$ we have 
\[
\mu(A)\geq \frac{1}{2}\quad \implies \quad
1-\mu(A+\beta \calz_{p}(\mu))\leq e^{-p}(1-\mu(A)), \qquad p \geq 2.
\]
\end{dfn}

\noindent All centered product log-concave measures satisfy the optimal concentration inequality with a universal constant $\beta$ (\cite{LW}). A natural conjecture (discussed in \cite{LW,LaIMA}) states that this is true also for nonproduct measures. However, one has to mention that it would imply (see Corollary 3.14. in \cite{LW}) the celebrated KLS conjecture (proposed in \cite{KLS} as a tool for proving efficiency of certain Metropolis type algorithms for computing volumes of convex sets) on the boundedness of the Cheeger constant for isotropic log-concave measures . It was shown in \cite{LaIMA} that every log-concave measure on $\R^n$ satisfies $CI(c\sqrt{n})$ with a universal constant $c$. The following corollary improves upon this bound. 

\begin{cor}
Every centered log-concave probability measure on $\er^n$ satisfies the optimal concentration inequality
with constant $\beta\leq Cn^{5/12}$.
\end{cor}

\begin{proof}
We follow the ideas expained after the proof of Proposition 7 in \cite{LaIMA}, but instead of Eldan's bound on
the Cheeger constant \cite{El} we use the recent result of Lee and Vempala  \cite{LV}.  

Since the concentration inequality is invariant with respect to linear transformation we may assume that $\mu$
is isotropic. Then in particular $\calz_p(\mu)\supset \calz_2(\mu)= B_2^n$. 

By Proposition 2.7 in \cite{LW} $\mathrm{CI}(\beta)$ may be equivalently stated as
\[
\forall_{p\geq 2}\ \forall_{A\in \mathcal{B}(\er^n)}\
\mu(A+\beta \calz_p(\mu))\geq \min\left\{\frac{1}{2},e^p\mu(A)\right\}.
\]
To show the above bound with $\beta=Cn^{5/12}$ we consider two cases.

i) If $2\leq p\leq n^{1/6}$ then
\begin{align*}
\mu(A+Cn^{5/12}\calz_p(\mu))
&\geq \mu(A+Cn^{1/4}pB_2^n)
\geq \min\left\{\frac{1}{2},e^{p}\mu(A)\right\},
\end{align*}
where the last inequality follows by the Lee-Vempala \cite{LV} $Cn^{1/4}$ bound on the Cheeger constant. 

ii) If $p\geq \max\{2,n^{1/6}\}$ then observe first that Theorem \ref{thm_main} (see Remark \ref{rem_zp}) yields
\[
\mu\left(2e^{3/2}\sqrt{\frac{n+p}{p}}\calz_p(\mu)\right)\geq 1-e^{-p}.
\] 
Therefore Lemma 9 in \cite{LaIMA} gives
\[
\mu(A+Cn^{5/12}\calz_p(\mu))\geq \mu\left(A+18e^{3/2}\sqrt{\frac{n+p}{p}}\calz_p(\mu)\right)
\geq \min\left\{\frac{1}{2},e^{p}\mu(A)\right\}.
\]
\end{proof}

\subsection*{$\textbf{p}$-summing norms of finite rank operators}

The theory of absolutely summing operators is an important part of the modern Banach space theory and found
numerous powerful applications in harmonic analysis, approximation theory, probability theory and operator theory
\cite{DJT}. 

Recall that a linear operator $T$ between Banach spaces $F_1$ and $F_2$ is $p$-summing
if there exists a constant $\alpha<\infty$, such that 
\[
\forall_{x_1,\ldots x_m\in F_1}\
\left(\sum_{i=1}^m\|Tx_i\|^p\right)^{1/p}
\leq \alpha \sup_{x^*\in F_1^*,\|x^*\|\leq 1}\left(\sum_{i=1}^m|x^*(x_i)|^p\right)^{1/p}.
\] 
The smallest constant $\alpha$ in the above inequality is called the $p$-summing norm of $T$ and denoted
by $\pi_p(T)$. For a Banach space $F$ by $\pi_p(F)$ we denote the $p$-summing constant of the identity map of $F$. 

It is well known that $\pi_p(F)<\infty$ if and only if $F$ is finite dimensional. Moreover
$\pi_2(F)=\sqrt{\dim F}$ (see \cite[Theorem 16.12.3]{Ga}). Summing constants of some finite dimensional
spaces were computed by Gordon in \cite{Go}. In particular he showed that 
\[
\pi_p(\ell_2^n)=\left(\Ex|U_1|^p\right)^{-1/p}\sim \sqrt\frac{n+p}{p}. 
\]

Immediate consequence of our main result is that up to a universal constant Hilbert space has the largest 
$p$-summing constant among all normed spaces of fixed dimension.

\begin{cor}
\label{cor:abssum1}
For any finite dimensional Banach space $F$ and $p\geq 2$ we have
\[
\pi_p(F)\leq 2\sqrt{2}\sqrt{\frac{\mathrm{dim}F+p}{p}}\leq C\pi_p(\ell_2^{\mathrm{dim}F}).
\]
\end{cor}

\begin{proof}
We apply Theorem \ref{thm_main} for random vectors uniformly distributed on finite subsets of $F$ 
and  $T$ the unit ball in $F^*$. 
\end{proof}

Using the ideal properties we get a bound for $p$-summing constant of finite rank operators.

\begin{cor}
Let $T$ be a finite rank linear operator between Banach spaces $F_1$ and $F_2$.
Then the $p$-absolutely summing constant of $T$ satisfies 
\[
\pi_p(T)\leq 2\sqrt{e}\sqrt{\frac{\mathrm{rk}(T)+p}{p}}\|T\|.
\]
\end{cor}

\begin{proof}
Let $F_3:=T(F_1)$, then $F_3$ is a subspace of $F_2$ of dimension $\mathrm{rk}(T)$. Observe that
$T=i\circ I\circ \tilde{T}$ where $\tilde{T}$ is $T$ considered as an operator between $F_1$ and $F_3$,
$I$ is the identity map on $F_3$ and $i$ is the embedding of $F_3$ into $F_1$. The operator ideal properties of the  $p$-summing norm and Corollary \ref{cor:abssum1} imply
\begin{align*}
\pi_p(T)
&\leq \|i\|\pi_p(I)\|\tilde{T}\|=\pi_p(F_3)\|T\|
\leq 2\sqrt{e}\sqrt{\frac{\mathrm{dim}(F_3)+p}{p}}\|T\|
\\
&=2\sqrt{e}\sqrt{\frac{\mathrm{rk}(T)+p}{p}}\|T\|.
\end{align*}
\end{proof}

\section{Open Questions}
\label{sec:openq}

Corollary 6 in \cite{La} states that for unconditional log-concave vectors in $\er^n$ and $2\leq p\leq n$
we have
\[
\frac{1}{C}\sqrt{\frac{n}{p}}\leq \Ex \|X\|_{\calz_p(X)}
\leq \left(\Ex  \|X\|_{\calz_p(X)}^{\sqrt{np}}\right)^{1/\sqrt{np}}\leq C\sqrt{\frac{n}{p}}.
\]
We do not know whether such bounds holds without unconditionality assumptions. We are only able 
to show the following weaker lower bound. Recall that the isotropic constant of a centered logconcave 
vector $X$ with density $g$ is defined as
\[
L_X:=(\sup_x g(x))^{1/n}(\mathrm{det\ Cov}(X))^{1/(2n)}.
\]
It is known that for all log-concave vectors $L_X\geq 1/C$, the famous open conjecture, due to Bourgain \cite{Bou}, states that $L_X\leq C$ (see \cite{AGM,BGVV} for more details and discussions of known upper bounds).

\begin{prop}
\label{prop:lowZp}
For any centered log-concave $n$-dimensional random vector with nondegenerate covariance matrix
we have
\[
\Ex \|X\|_{\calz_p(X)}\geq \frac{1}{CL_X}\sqrt{\frac{n}{p}} \quad \mbox{ for }1\leq p\leq n,
\]
where $L_X$ is the isotropic constant of $X$.
\end{prop}

\begin{proof}
Since the assertion is linearly invariant we may and will assume that $X$ is isotropic, i.e. it has the identity
covariance matrix. The density of $X$ is then bounded by $L_X^n$, hence
\begin{align*}
\Pr\left(\|X\|_{\calz_p(X)}\leq t\sqrt{n/p}\right)
&=\Pr\left(X\in t\sqrt{n/p}\calz_p(X)\right)
\\
&\leq L_X^n\mathrm{vol}\left(t\sqrt{n/p}\calz_p(X)\right)\leq (C_1tL_X)^n,
\end{align*}
where the last estimate follows by the Paouris \cite{Pa} bound on the volume of $\calz_p$-bodies 
(see also \cite[Theorem 5.1.17]{BGVV}).

Thus
\[
\Ex\|X\|_{\calz_p(X)}
\geq \frac{1}{2C_1L_X}\sqrt\frac{n}{p}\Pr\left(\|X\|_{\calz_p(X)}> \frac{1}{2C_1L_X}\sqrt\frac{n}{p}\right)
\geq \frac{1}{4C_1L_X}\sqrt\frac{n}{p}.
\]
\end{proof}

In the last years it was showed that various constants related to the $n$-dimensional log-concave measures
(isotropic constant, Cheeger constant, thin-shell constant) are bounded by $Cn^{1/4}$. We think that the same 
should be true for the $\mathrm{CI}$ constant.

Finally, we do not know whether the optimal constants in \eqref{eq:momZp} are attained for rotationally invariant
random vectors or, in other words, whether conditions in Lemma \ref{lem:main} hold with 
$C_{n,p}=(\Ex|U_1|^p)^{1/p}$, where $U_1$ is the first coordinate of the random vector $U$ uniformly distributed on the unit sphere. Equivalently one may ask whether for any finite dimensional Banach space $F$ one has
$\pi_p(F)\leq \pi_p(\ell_2^{\dim F})$.


\begin{thebibliography}{1}

\bibitem{ALOPT} R.~Adamczak, R.~Lata{\l}a, A.~Litvak, K.~Oleszkiewicz, A.~Pajor and N.~Tomczak-Jaegermann, 
\textit{A short proof of Paouris' inequality}, 
Can. Math. Bull. \textbf{57} (2014), 3--8.

\bibitem{AGM} S.~Artstein-Avidan, Shiri, A.~Giannopoulos and V.~D.~Milman,
\textit{Asymptotic Geometric Analysis. Part I},
Mathematical Surveys and Monographs \textbf{202}, American Mathematical Society, Providence, RI, 2015. 

\bibitem{B75} C.~Borell, 
\textit{The Brunn-Minkowski inequality in Gauss space}, 
Invent. Math. \textbf{30} (1975), 207--216.

\bibitem{Bou}  J.~Bourgain,  
\textit{On high-dimensional maximal functions associated to convex bodies,} 
Amer. J. Math. \textbf{108} (1986), 1467--1476.

\bibitem{BGVV} S.~Brazitikos, A.~Giannopoulos, P.~Valettas and B.~H.~Vritsiou, 
\textit{Geometry of isotropic convex bodies},  Mathematical Surveys and Monographs \textbf{196}, 
American Mathematical Society, Providence, RI, 2014.

\bibitem{DSD} S.~Datta,  H.~Stephen,  and  C.~Douglas,  
\textit{Geometry  of  the  Welch  bounds}, 
Linear  Algebra  and  its  Applications \textbf{437} (2012), 2455--2470.

\bibitem{DJT} J.~Diestel, H.~Jarchow and A.~Tonge, 
\textit{Absolutely summing operators}, 
Cambridge Studies in Advanced Mathematics \textbf{43}, Cambridge University Press, Cambridge, 1995. 

\bibitem{El} R.~ Eldan,
\textit{Thin shell implies spectral gap up to polylog via a stochastic localization scheme},
Geom. Funct. Anal. \textbf{23} (2013), 532--569.

\bibitem{Ga} D.J.H.~Garling,
\textit{Inequalities: a journey into linear analysis}, 
Cambridge University Press, Cambridge, 2007.

\bibitem{Go} Y.~Gordon,
\textit{On $p$-absolutely summing constants of Banach spaces},
Israel J. Math. \textbf{7} (1969) 151–-163.


\bibitem{KLS} R.~Kannan, L.~Lov\'asz and M.~Simonovits, 
\textit{Isoperimetric problems for convex bodies and a localization lemma}, 
Discrete Comput. Geom. \textbf{13} (1995), 541-–559.

\bibitem{La} R.~Lata{\l}a,
\textit{On ${\mathcal Z}_p$-norms of random vectors}, 
Zap. Nauchn. Sem. S.-Peterburg. Otdel. Mat. Inst. Steklov. (POMI) \textbf{457} (2017), 
Veroyatnost\' i Statistika 25, 211--225.

\bibitem{LaIMA} R.~Lata{\l}a,
\textit{On some problems concerning log-concave random vectors}, Convexity and Concentration, 525--539, 
The IMA Volumes in Mathematics and its Applications \textbf{161}, Springer 2017.


\bibitem{LS2}  R.~Lata{\l}a and M.~Strzelecka, 
\textit{Weak and strong moments of $l_r$-norms of log-concave vectors}, 
Proc. Amer. Math. Soc. \textbf{144} (2016), 3597--3608.

\bibitem{LW}  R.~Lata{\l}a and J.~O.~Wojtaszczyk, 
\textit{On the infimum convolution inequality}, 
Studia Math. \textbf{189} (2008), 147--187.

\bibitem{Le} M.\ Ledoux, 
\textit{The concentration of measure phenomenom}, Mathematical Surveys and Monographs \textbf{89}, 
American Mathematical Society, Providence, RI, 2001.

\bibitem{LV} Y.T.~Lee, and S. Vempala,
\textit{Eldan's stochastic localization and the {KLS} hyperplane conjecture: an improved lower bound for expansion},
58th Annual IEEE Symposium on Foundations of Computer Science -- FOCS 2017, 998--1007, IEEE Computer Soc., Los Alamitos, CA, 2017. 

\bibitem{Pa} G.~Paouris, 
\textit{Concentration of mass on convex bodies}, 
{Geom. Funct. Anal.} \textbf{16} (2006), 1021--1049.

\bibitem{peng} I.~Peng, S.~Waldron, 
\textit{Signed frames and Hadamard products of Gram matrices}, 
Linear Algebra Appl. \textbf{347} (2002), 131-–157.


\bibitem{SC74} V.~N.~Sudakov and B.~S.~Tsirelson, 
\textit{Extremal properties of half-spaces for spherically invariant measures}, 
Zap. Naucn. Sem. Leningrad. Otdel. Mat. Inst. Steklov. (LOMI) \textbf{41} (1974) 14--24, 
Problems in the theory of probability distributions, II.

\bibitem{Ta} M.~Talagrand,
\textit{A new isoperimetric inequality and the concentration of measure phenomenon},
in: Israel Seminar (GAFA), Lecture Notes in Math. \textbf{1469}, 94--124, Springer, Berlin 1991

\bibitem{TaMon} M.~Talagrand, 
\textit{Upper and lower bounds for stochastic processes}, 
Ergebnisse der Mathematik und ihrer Grenzgebiete \textbf{60}, Springer, Heidelberg, 2014.

\bibitem{Welch} L.~Welch, 
\textit{Lower  bounds  on  the  maximum  cross  correlation  of  signals}, 
IEEE  Transactions  on  Information  Theory \textbf{20} (1974),  397--399.

\end{thebibliography}
\end{document}